\documentclass[12pt]{article}

\usepackage{amsmath}
\usepackage{amsfonts}
\usepackage{amssymb}
\usepackage[colorlinks=true,linkcolor=black,citecolor=blue,urlcolor=blue,]{hyperref}
\newtheorem{theorem}{Theorem}[section]
\newtheorem{lemma}{Lemma}[section]

\newtheorem{remark}{Remark}[section]
\newenvironment{proof}{\noindent{\bf Proof:}}{\hfill\fbox{}\vspace*{1mm}}

\begin{document}
	
	\title{Toeplitz Operators and Berezin-type Operators on Different Bergman Spaces}

	\author{ Lvchang Li, Haichou Li}
\author{Lvchang Li \thanks
	{College of Mathematics and Informatics,
		South China Agricultural University,
		Guangzhou,
		510640,
		China
		Email: 20222115006@stu.scau.edu.cn.},\
	Haichou Li \thanks{ Corresponding author,
		College of Mathematics and Informatics,
		South China Agricultural University,
		Guangzhou,
		510640,
		China
		Email: hcl2016@scau.edu.cn.  Li is supported by NSF of China (Grant No. 12326407 and 12071155). }
}
	
	\date{}
	\maketitle
	\begin{center}
		\begin{minipage}{120mm}
			\begin{center}{\bf Abstract}\end{center}
			{In the present paper, we study the boundedness and compactness of Toeplitz operators and Berezin-type operators between different weighted Bergman spaces over tubular domains in $\mathbb{C}^n$. We establish their connection with Carleson measures and provide some characterizations.}		
			
			{\bf Key words}:\ \ Bergman space; Toeplitz operator; Berezin-type operator; Carleson measure
		\end{minipage}
	\end{center}

\section{Introduction}
\ \ \ \
Carleson measures are a powerful tool in modern complex analysis, initially introduced by Carleson \cite{carleson1958interpolation} while studying the Corona problem on the Hardy space structure of the unit disk.
Subsequently, it became the most essential tool in the study of operator theory on function spaces. Regarding the further applications of Carleson measures in operator theory on function spaces, refer to \cite{abate2011carleson,cima1982carleson,hastings1975carleson,luecking1983technique,li2024carleson}. 

It is an interesting topic to study the analytic properties of operators from one function space to another function space. In particular, it is also meaningful to study the boundedness and compactness of parametric operators from one weighted Bergman space to another weighted Bergman space. In the present paper, We study the analytic properties of two kinds of operators on the Bergman spaces over tubular domains. Specifically, we study the relationship between the boundedness and compactness of the operators and the Carleson measures. Finally, the characterization between them is given.

The main motivation of this work comes from the work of many scholars. In some simple domains, such as the unit ball, Zhao et al.\cite{pau2015carleson}studied the boundedness and compactness of Toeplitz operators on this domain, and connected it with Carleson measure, and gave an equivalent characterization. Subsequently, zhou et al.\cite{prǎjiturǎ2023berezin} improved the results of zhao et al.\cite{pau2015carleson} on this basis, and studied a class of operators similar to Toeplitz operators, later called Berezin-type operators, which evolved from Berezin transform. At the end of the article, Zhou et al.\cite{prǎjiturǎ2023berezin} further discussed the case when the index is infinite. For the Toeplitz operator theory on the unit disk and the unit ball, see \cite{zhu2005spaces,zhu2007operator}. On bounded strongly pseudoconvex domains, Abate et al.\cite{abate2019toeplitz,abate2012toeplitz,abate2011carleson} also studied similar problems. 

The previous studies are all on bounded domains, and some scholars have also carried out related research on unbounded domains of interest to the author of this paper. Si et al.\cite{si2022carleson} studied the relationship between the Toeplitz operators and the Carleson measures on the Bergman spaces on the Siegel upper half spaces. In this paper, we will study the tubular domains. Deng et al.\cite{deng2021reproducing} provided the reproducing kernels of Bergman spaces on tubular domains by a new method, and Liu et al.\cite{jiaxin2023bergman} presented fundamental properties of the Bergman metric on these domains, laying the groundwork for subsequent theoretical studies. The first author Li \cite{Li} studied the Toeplitz operators on the weighted Bergman spaces over the tubular domains, and established the relationship between the boundedness and compactness of the operators and the Carleson measures. This paper is the subsequent work of \cite{Li}. The operators studied in this paper are Berezin-type operators and Toeplitz operators with parameters. The Berezin-type operator is derived from the Berezin transform. The Berezin transform is an important tool for studying operator theory on function spaces. For details, please refer to \cite{zhu2007operator,Berezin transform}.

The organization of the paper is as follows: The second part introduces some basic terms; The third part gives some instrumental lemmas; The fourth part gives the characterization of Carleson measures and vanishing Carleson measures on the tubular domains; In the last part, we discuss Toeplitz operators and Berezin type operators. We link their boundedness and compactness with Carleson measure, and obtain the main results, including Theorem \ref{dth 1}, \ref{dth 2} and \ref{dth 3}.

\section{Preliminaries}
\ \ \ \
Let $\mathbb{C}^n$ be the $n$ dimensional complex Euclidean space. For any two points $z=\left( z_1,\cdots ,z_n \right) $ and $w=\left( w_1,\cdots ,w_n \right) $ in $\mathbb{C}^n,$ we write\[z\cdot \bar{w}:=z_1\bar{w}_1+\cdots +z_n\bar{w}_n,\] \[z^2=z\cdot z:=z_{1}^{2}+z_{2}^{2}+\cdots +z_{n}^{2}\] and \[\left| z \right|:=\sqrt{z\cdot \bar{z}}=\sqrt{\left| z_1 \right|^2+\cdots +\left| z_n \right|^2}.\]  

The set $\mathbb{B}_n=\left\{ z\in \mathbb{C}^n:\left| z \right|<1 \right\}$ will be called the unit ball of $\mathbb{C}^n$. 

The tubular domain $T_B$ of $\mathbb{C}^n$ with base $B$, is  defined as follows:

$$ T_B=\left\{ z=x+iy \in \mathbb{C}^n | x\in \mathbb{R}^n, \  y\in B \subseteq  \mathbb{R}^n  \right\}, $$   
where\[B=\left\{ \left( y',y_n \right)=( y_1,\cdots ,y_{n-1}, y_n) \in \mathbb{R}^n\left| y'^2:=y_{1}^{2}+\cdots +y_{n-1}^{2}<y_n \right. \right\}.\] 

We define the spaces $L_{\alpha}^{p}\left( T_B \right)$, which is composed of all Lebesgue measurable functions $f$ on $T_B$, and its norm \[\lVert f \rVert _{p,\alpha}=\left\{ \int_{T_B}{\left| f\left( z \right) \right|^p dV_\alpha\left( z \right)} \right\} ^{\frac{1}{p}}\] is finite, where $dV_\alpha(z)=( y_n-\left| y' \right|^2)^\alpha dV(z)$, $\alpha>-1$, $dV(z)$ denotes the Lebesgue measure on $\mathbb{C}^n$. 

The space $L^\infty(T_B)$ consists of all bounded Lebesgue measurable functions $f$ on $T_B$ such that $$\lVert f \rVert _{\infty}=\mathop{\text{ess\ sup}}_{z\in T_B}\left| f\left( z \right) \right|$$ is finite.

The Bergman spaces $A_{\alpha}^{p}\left( T_B \right)$ on tube $T_B$  is a set composed of all holomorphic functions in $L_{\alpha}^{p}\left( T_B \right)$. We know that when $1\le p<\infty $ the space $A_{\alpha}^{p}\left( T_B \right)$ is a Banach space with the norm $\lVert \cdot \rVert _{p,\alpha}.$ In particular, when $p=2,$ $A_{\alpha}^{2}\left( T_B \right)$ is a Hilbert space. 

 An very important orthogonal projection from $L_{\alpha}^{2}\left( T_B \right)$ to $A_{\alpha}^{2}\left( T_B \right)$ is the following integral operator:\[P_{\alpha}f\left( z \right) =\int_{T_B}{K_\alpha\left( z,w \right) f\left( w \right) dV_\alpha\left( w \right)},\] with the Bergman kernel  \[K_{\alpha}\left( z,w \right) =\frac{2^{n+1+2\alpha}\varGamma \left( n+1+\alpha \right)}{\varGamma \left( \alpha +1 \right) \pi ^n}\left( \left( z'-\overline{w'} \right) ^2-2i\left( z_n-\overline{w_n} \right) \right) ^{-n-\alpha -1}.\]

For convenience, we introduce the following notations: \[\rho \left( z,w \right) =\frac{1}{4}\left( \left( z'-\overline{w'} \right) ^2-2i\left( z_n-\overline{w_n} \right) \right)\] and let $\rho \left( z \right) :=\rho \left( z,z \right) =y_n-y'^2.$ 

With the above notion $\rho \left( z,w \right)$, the weighted Bergman kernel of $T_B$ becomes
\[K_{\alpha}\left( z,w \right) =\frac{\varGamma \left( n+\alpha +1 \right)}{2^{n+1}\pi ^n\varGamma \left( \alpha +1 \right) \rho \left( z,w \right) ^{n+\alpha +1}}.\]

We state some explanations regarding the boundary of tubes $T_B$.

Recall $\rho \left( z \right) =y_n-y'^2$ and let $\partial T_B:=\left\{ z\in \mathbb{C}^n\,\,: \rho \left( z \right) =0 \right\}$ denote the boundary of $T_B.$ Then $$\widehat{T_B}:=T_B\cup \partial T_B\cup \left\{ \infty \right\}$$ is the one-point compactification of $T_B.$

Also, let $\partial \widehat{T_B}:=\partial T_B\cup \left\{ \infty \right\} .$ Thus, $z\rightarrow \partial \widehat{T_B}$ 
means $\rho \left( z \right) \rightarrow 0$ or $\left| z \right|\rightarrow \infty .$

We review the Bergman metric ball on domains in $\mathbb{C}^n$. 

We will use the important  transform $\varPhi :\mathbb{B}_{n}\rightarrow T_{B}$ given by
$$\varPhi(z)=\left ( \frac{\sqrt{2}z'}{1+z_{n}}, \ i\frac{1-z_n}{1+z_n}-i\frac{z' \cdot z'}{(1+z_n)^2} \right ), \ z\in \mathbb{B}_n$$
and it is not hard to calculate that 
$$\varPhi^{-1}(w)=\left ( \frac{2iw'}{i+w_{n}+\frac{i}{2}w'\cdot w'},\  \frac{i-w_n-\frac{i}{2}w'\cdot w'}{i+w_{n}+\frac{i}{2}w'\cdot w'} \right ), \ w\in T_{B}.$$
The mapping $\varPhi$ is a biholomorphic map from $\mathbb{B}_n$ to $T_B$ and a key tool for this paper.

In Krantz's book, there is the following proposition \cite[proposition 1.4.12]{krantz2001function}:

Let $\varOmega _1,\ \varOmega _2\subseteq \mathbb{C}^n$ be domains and $ f:\varOmega _1\rightarrow \varOmega _2$ a biholomorphic mapping. Then $f$ induces an isometry of Bergman metrics:$$\beta _{\varOmega _1}\left( z,w \right) =\beta _{\varOmega _2}\left( f\left( z \right) ,f\left( w \right) \right) $$ for all $z,w\in \varOmega _1$.

Hence, taking $\varOmega _1=T_B$ and $\varOmega _2=\mathbb{B}_n$, we have:$$
\beta _{T_B}\left( z,w \right) =\beta _{\mathbb{B}_n}\left( \Phi ^{-1}\left( z \right) ,\Phi ^{-1}\left( w \right) \right) =\tanh ^{-1}\left( \left| \varphi _{\Phi ^{-1}\left( z \right)}\left( \Phi ^{-1}\left( w \right) \right) \right| \right) .
$$
A computation shows that
$$
\beta_{T_B}(z, w)=\tanh ^{-1} \sqrt{1-\frac{\rho(z) \rho(w)}{|\rho(z, w)|^2}} .
$$
Let $D\left( z,r \right) $ denote the Bergman metric ball at $z$ with radius $r$, that is 
\[D(z,r)=\left\{w \in T_B:  \beta(z,w)<r\right\}.\]

Let $\mathcal{M}_+$ be the set of all positive Borel measure $\mu$ such that \[\int_{T_B}{\frac{d\mu \left( z \right)}{\left| \rho \left( z,\mathbf{i} \right) \right|^t}}<\infty \] for some $t>0$, where $\mathbf{i}=(0',i).$

Next, we introduce the two operators studied in the present paper.

For $s,t>0$, we denote $$
B_{s,t}\left( \mu \right) \left( z \right) :=\rho \left( z \right) ^t\int_{T_B}{\frac{1}{|\rho \left( z,w \right) |^{\left( n+1+\alpha  \right) s+t}}}d\mu \left( w \right) .
$$
This is called a Berezin-type transform for the measure $\mu$, where $\alpha > -1$.
Notice that $B_{1,n+\alpha+1}$  is the canonical Berezin transform of $\mu$ except a constant.

Let $\xi  >-1$ and $ \mu \in \mathcal{M}_+ $  a positive Borel measure on $T_B$. For $f\in H(T_B)$, where $H(T_B)$ is the space consisting of all holomorphic functions on $T_B$, we consider the following sublinear operator:  \[ B_{\mu}^{\xi}f\left( z \right) =\int_{T_B}{\frac{|f\left( w \right) |}{\left| \rho \left( z,w \right) \right|^{n+1+\xi}}}\text{d}\mu \left( w \right), \  z\in T_B.\]
We refer to $ B_{\mu}^{\xi}$ as a Berezin-type operator.

The Toeplitz operator $T_{\mu}^{\xi}$ is defined by:  \[T_{\mu}^{\xi}f\left( z \right) =\int_{T_B}{\frac{f\left( w \right)}{\rho \left( z,w \right) ^{n+1+\xi}}}\text{d}\mu \left( w \right), \  z\in T_B.\]
It is clear that $\left| T_{\mu}^{\xi}f \right|\le B_{\mu}^{\xi}f$ for $f \in H(T_B)$. Therefore, The boundedness of Berezin-type operators implies the boundedness of Toeplitz operators.

For any positive numbers $\lambda > 0$ and $\alpha > -1$, a positive Borel measure $\mu$ is considered a $(\lambda,\alpha)$-Carleson measure, if for all positive integers $p$ and $q$ where $q/p = \lambda$, there exists a positive constant $C > 0$ such that
$$
\int_{T_B} \left| f(z) \right|^q \, d\mu(z) \leq C \lVert f \rVert_{p,\alpha}^q
$$
for every $f \in A_\alpha^{p}(T_B)$.

A positive Borel measure $\mu$ as a vanishing $(\lambda,\alpha)$-Carleson measure, if for any positive numbers $p$ and $q$ such that $q/p=\lambda$, and for any bounded sequence ${f_k}$ in $A_\alpha^{p}(T_B)$ that converges uniformly to $0$ on every compact subset of $T_B$, we have \[\lim_{k\rightarrow \infty} \int_{T_B}{\left| f_k \right|^q d\mu}=0.\]
We also denote by 
$$
\lVert \mu \rVert _{\lambda ,\alpha}=\mathop{\text{sup}}_{f\in A_{\alpha}^{p}\left( T_B \right) ,\lVert f \rVert _{p,\alpha}\le 1}\int_{T_B}{\left| f\left( z \right) \right|^qd\mu \left( z \right)}.
$$

Finally, we recall the definition of a compact operator: Given two Banach spaces X and Y, a bounded linear operator is called compact if it maps any bounded set in X to a set that is relatively compact in the topology of Y. When the Banach space X is reflexive, a compact operator on X is equivalent to being completely continuous. Complete continuity refers to an operator that maps weakly convergent sequences to sequences that converge in norm.
For sublinear operators, the boundedness and compactness of operators are also defined according to the boundedness and compactness of linear operators

Throughout the paper we use C to denote positive constants whose value may change from line to line but does not depend on the functions being considered. The notation $A\lesssim B$ means that there is a positive constant C such that $A\le CB$, and the notation $A\simeq B$ means that $A\lesssim B$ and $B\lesssim A$.

\section{Main lemmas}
\ \ \ \
To prove our main results, we need the following key lemmas, where Lemmas \ref{lem duliang}-\ref{lem jifendengshi} are from \cite{jiaxin2023bergman}. They play a crucial role as instrumental lemmas in the present paper.

The following lemma \ref{lem duliang} is the so-called Bergman decomposition on tubular domains. It plays an important role in the topic of operator theory on function spaces.
\begin{lemma}\label{lem duliang}
	There exists a positive integer $N$ such that for any $0<r\le1$ we can find a sequence $\{a_k\}$ in $T_B$ with the following properites:
	
	$(1)$ $T_B=\bigcup_{k=1}^{\infty}{D\left( a_k,r \right)};$
		
	$(2)$ The sets $D(a_k,r/4)$ are mutually disjoint;

	$(3)$ Each point $z\in T_B$ belongs to at most $N$ of the sets $D(a_k,2r)$.

\end{lemma}

Any sequence $\{a_k\}$ satisfying the conditions of the above lemma is called a lattice.
\\

The following lemmas \ref{lem dengjia} and \ref{lem ceByuanpan} give the control of the size of $\left| \rho \left( z,w \right) \right|$ in the Bergman ball $D\left( z,r \right)$ and the size of the Bergman ball $D\left( z,r \right)$ under the Euclidean metric, respectively.
\begin{lemma}\label{lem dengjia}
	For any $r>0$, the inequalities 
$$\left| \rho \left( z,u \right) \right|\simeq \left| \rho \left( z,v \right) \right|$$
	hold for all $z,u,v\in T_{B} $ with $\beta (u,v) <r .$
\end{lemma}

\begin{lemma}\label{lem ceByuanpan}
	For any $z\in T_B$ and $r>0$ we have \[V_{\alpha}\left( D\left( z,r \right) \right) \simeq \rho \left( z \right) ^{n+\alpha +1}.\]
\end{lemma}

The following lemma \ref{lem jifendengshi} plays a key role in the integral operation in this paper.
\begin{lemma}\label{lem jifendengshi}
	Let $r, s>0, t>-1$ and $r+s-t>n+1$, then 
	\[  \int_{T_{B}}\frac{{\rho (w)}^{t}}{{\rho (z,w)}^{r}{\rho (w,u)}^{s}}dV(w)=\frac{C_{1}(n,r,s,t)}{\rho (z,u)^{r+s-t-n-1}}        \]
	for all $z,u \in T_{B}$, where
	$$C_{1}(n,r,s,t)=\frac{2^{n+1}{\pi}^{n}\Gamma (1+t)\Gamma (r+s-t-n-1)}{\Gamma (r)\Gamma (s)}.$$
	In particular, let $s,t\in \mathbb{R}$, if $t>-1,s-t>n+1$, then \[\int_{T_B}{\frac{\rho \left( w \right) ^t}{\left| \rho \left( z,w \right) \right|^s}}dV\left( w \right) =\frac{C_1\left( n,s,t \right)}{\rho \left( z \right) ^{s-t-n-1}}.\] Otherwise, the above equation is infinity.
\end{lemma}

The following lemma is commonly encountered in the operator theory on function spaces topic, illustrating that the growth of functions in Bergman spaces is controlled.
\begin{lemma}\label{lem ci tiao he fang da}
	
	On the Bergman space $A_{\alpha}^{p}\left( T_B \right)$, every valuation functional is a bounded
	linear functional. More specifically, each function $f\in T_B$ has the property 
	\[\left| f\left( z \right) \right|^p\le \frac{C}{\rho \left( z \right) ^{n+\alpha+1}}\int_{D\left( z,r \right)}{\left| f\left( w \right) \right|^pdV_\alpha\left( w \right)}.\]
\end{lemma}

The following lemma \ref{lem zitui} is a utility lemma, and its proof can be found in \cite{Li}.

\begin{lemma}\label{lem zitui}
	As $\left| z \right|$ approaches infinity, $\rho \left( z,\mathbf{i} \right)$ also tends to infinity.
\end{lemma}

We provide several lemmas related to the properties of reproducing kernels, as seen in Lemma \ref{lem zai shng he de da xiao} and Lemma \ref{lem ruo shou lian}.
\begin{lemma}\label{lem zai shng he de da xiao}
	Let $1<p<\infty,\alpha>-1$ for each $z>0$, the Bergman kernel function $K_\alpha(z,w)$ is in $A_\alpha^p(T_B),$ and its norm is $C\rho \left( z \right) ^{-\left( n+\alpha+1 \right) /p'}$, where  $p'=p/(p-1)$ and $C$ is a positive constant only depending on $n,\alpha,$ and $p.$
\end{lemma}

\begin{lemma}\label{lem ruo shou lian}
	For $1<p<\infty$ and $\alpha>-1$, we have $K_{\alpha,z}\lVert K_{\alpha,z} \rVert _{p,\alpha}^{-1}\rightarrow 0$ weakly in $A_{\alpha}^{p}\left( T_B \right)$ as $z\rightarrow \partial \widehat{T_B}.$ 
\end{lemma}

\section{Carleson measures on $T_B$}
\ \ \ \
Since the Carleson measure is a very important tool in operator theory, we first need to characterize the Carleson measures and the vanishing Carleson measures. Carleson measures and vanishing Carleson measures on the unit disk and unit ball have numerous characterizations, as detailed in \cite{zhu2005spaces,zhu2007operator}. We now provide characterizations of Carleson measures and vanishing Carleson measures on tubular domains, which will serve as  basis results for our subsequent theorems.

In order to get the characterization of Carleson measures, we need the help of the following lemma \ref*{lem xin 3}.
\begin{lemma}\label{lem xin 3}
	For $1 \leq \lambda<\infty$, $-1<\alpha<\infty$, and $\mu$ being a $(\lambda, \alpha)$-Carleson measure on $T_B$, satisfying $\mu \left( D\left( z,r \right) \right) \lesssim \rho \left( z \right) ^{\left( n+1+\alpha \right) \lambda}$ for any Bergman ball $D\left( z,r \right)$, the following inequality holds for any $f \in H\left(T_B\right)$ and $0<p<\infty$:
	$$
	\int_{T_B}|f(z)|^p d \mu(z) \lesssim \int_{T_B}|f(z)|^p\rho (z)^{(n+1+\alpha) \lambda-(n+1)} d V(z) .
	$$
\end{lemma}

\begin{proof}
	
	By Lemma \ref{lem ci tiao he fang da}, for $0 < r < 1$, we have
	$$
	|f\left( z \right) |^p\le \frac{C}{\rho \left( z \right) ^{n+1}}\int_{D\left( z,r \right)}{|}f\left( w \right) |^pdV\left( w \right)  .
	$$
	Hence, using Fubini's theorem along with Lemma \ref{lem dengjia}, we obtain
	$$
	\begin{aligned}
	\int_{T_B}{|}f\left( z \right) |^pd\mu \left( z \right)& \lesssim \int_{T_B}{\frac{1}{\rho \left( z \right) ^{n+1}}}\int_{D\left( z,r \right)}{|}f\left( w \right) |^pdV\left( w \right) d\mu \left( z \right)\\
	&=\int_{T_B}{|}f\left( w \right) |^p\int_{D\left( w,r \right)}{\frac{d\mu \left( z \right)}{\rho \left( z \right) ^{n+1}}}dV\left( w \right)\\
	&\lesssim \int_{T_B}{|}f\left( w \right) |^p\int_{D\left( w,r \right)}{\frac{d\mu \left( z \right)}{\rho \left( w \right) ^{n+1}}}dV\left( w \right)\\
	&=\int_{T_B}{|}f\left( w \right) |^p\frac{\mu \left( D\left( w,r \right) \right)}{\rho \left( w \right) ^{n+1}}dV\left( w \right)\\
	&\lesssim \int_{T_B}{|}f\left( w \right) |^p\rho \left( w \right) ^{\left( n+1+\alpha \right) \lambda -\left( n+1 \right)}dV\left( w \right) .
	\end{aligned}
	$$
	This completes the proof of Lemma \ref{lem xin 3}.
\end{proof}

\begin{theorem}\label{xt1}
	Suppose $1 \leq \lambda < \infty $ and $-1 < \alpha < \infty $, the followiong conditions are equivalent:
	\\
		
		$(1)$ $\mu$ is a $(\lambda , \alpha )$-Carleson measure.
		
		$(2)$ For any real number $r$ with $0<r<1$ and any $a \in T_B$\[\mu \left( D\left( a,r \right) \right) \lesssim \rho \left( a \right) ^{\left( n+1+\alpha \right) \lambda}.\]
		
		$(3)$ For some (every) $t>0$, the Berezin-type transform of $\mu$ \[B_{\lambda,t}\left( \mu \right) \in L^{\infty}\left( T_B \right).\]

\end{theorem}

\begin{proof}
	
	$\left( 1 \right) \Rightarrow \left( 2 \right) 	$
	First assume that $\mu$ is a  $(\lambda , \alpha )$-Carleson measure for $A_{\alpha}^{p}(T_B)$. There is a constant $C>0$ such that \[\int_{T_B}{\left| f\left( z \right) \right|^q}d\mu \left( z \right) \le C\lVert f \rVert _{p,\alpha}^{q}\] for all $f \in A_{\alpha}^{p}(T_B)$. 
	
	In particular, if $a \in T_B$ and \[g(w)=\left( \frac{\rho \left( a \right) ^{n+1+\alpha}}{\left| \rho \left( w,a \right) \right|^{2\left( n+1+\alpha \right)}} \right) ^{1/p},  \  z \in T_B,\] then $ \lVert g\left( w \right) \rVert _{p,\alpha} $ is a constant, so \[\int_{D\left( a,r \right)}{\left| g\left( w \right) \right|^q}d\mu \left( w \right) \le \int_{T_B}{\left| g\left( w \right) \right|^q}d\mu \left( w \right) \le C.\] 
	
	According to Lemma \ref{lem dengjia}, $\left|\rho (w,a)\right|$ is comparable to $\left|\rho (a)\right|$ when $w \in D(a,r)$. we have $\mu \left( D\left( a,r \right) \right) \lesssim \rho \left( a \right) ^{\left( n+1+\alpha \right) \lambda}$.
	\\
	
	$\left( 2 \right) \Rightarrow \left( 3 \right) 	$
	It follows from Lemma \ref{lem xin 3} that \[
	B_{\lambda,t}\left( \mu \right) \left( z \right) \lesssim \rho \left( z \right) ^t\int_{T_B}{\frac{1}{|\rho \left( z,w \right) |^{\left( n+1+\alpha \right) \lambda +t}}}\rho \left( w \right) ^{\left( n+1+\alpha\right) \lambda -\left( n+1 \right)}dV\left( w \right) .        \]
	
	By calculation, we find that the parameters in the above expression satisfy the conditions of Lemma \ref{lem jifendengshi}, so the above expression is less than a constant.
	\\

	$\left( 3 \right) \Rightarrow \left( 1 \right) 	$
	Assume $(3)$ is ture. For any $f \in A_{\alpha}^{p}\left( T_B \right) $, we have
	$$
	\int_{T_B}{|}f\left( z \right) |^qd\mu \left( z \right) \le \sum_{k=1}^{+\infty}{\int_{D\left( a_k,r \right)}{|}}f|^qd\mu \le \sum_{k=1}^{+\infty}{\mu}\left( D\left( a_k,r \right) \right) \mathop{\text{sup}}_{z\in D\left( a_k,r \right)}|f\left( z \right) |^q
	.$$ for all $k \ge 1.$
	
	It follows from Lemma \ref{lem ci tiao he fang da} that
	$$
	\sup _{z \in D\left(a_k, r\right)}|f(z)|^p \lesssim \frac{1}{\rho\left(a_k\right)^{n+\alpha+1}} \int_{D\left(a_k, 2 r\right)}|f(w)|^p d V_{\alpha}(w).
	$$
	
	Then we obtain
	$$
	\int_{T_B}{|}f\left( z \right) |^qd\mu \left( z \right) \lesssim \sum_{k=1}^{\infty}{\frac{\mu \left( D\left( a_k,r \right) \right)}{\rho \left( a_k \right) ^{\left( n+\alpha +1 \right) \lambda}}}\left( \int_{D\left( a_k,2r \right)}{|}f\left( z \right) |^pdV_{\alpha}\left( z \right) \right) ^{\lambda},
	$$ where $\lambda=q/p$.
	\\
	
	From $(3)$, when $z\in D\left( a_k,r \right) $ holds, $\mu \left( D\left( a_k,r \right) \right)/\rho \left( a_k \right) ^{\left( n+\alpha +1 \right) \lambda}\lesssim C$ is true.
	\\
	
	Since $\lambda>1$, it follows that
	$$
	\begin{aligned}
	\int_{T_B}{|}f\left( z \right) |^qd\mu \left( z \right) &\lesssim \sum_{k=1}^{+\infty}{\left( \int_{D\left( a_k,2r \right)}{|}f\left( z \right) |^pdV_{\alpha}\left( z \right) \right) ^{\lambda}}\\
	&\lesssim \left( \sum_{k=1}^{+\infty}{\int_{D\left( a_k,2r \right)}{|}}f\left( z \right) |^pdV_{\alpha}\left( z \right) \right) ^{\lambda}\\
	&\le N^{\lambda}\lVert f \rVert _{p,\alpha}^{q},\\
	\end{aligned}
	$$
	where $N$ is as in Lemma \ref{lem duliang}.  The proof is complete.	
\end{proof}

\begin{theorem}\label{xt2}
	Suppose $1 \leq \lambda<\infty$ and $-1<\alpha<\infty$, the following statements are equivalent:
	\\
	
	$(1)$ $\mu$ is a vanishing $(\lambda, \alpha)$-Carleson measure.
	
	$(2)$ For some(any) $t>0$
	$$
	\lim_{a\rightarrow \partial \widehat{T_B}} \int_{T_B}{\frac{\rho \left( a \right) ^t}{|\rho \left( z,a \right) |^{\left( n+1+\alpha \right) \lambda +t}}}d\mu \left( z \right) =0.
	$$
	
	$(3)$ For any real number $r$ with $0<r<1$ and any $a \in T_B$
	$$
	\lim_{a\rightarrow \partial \widehat{T_B}} \frac{\mu \left( D\left( a,r \right) \right)}{\rho \left( a \right) ^{\left( n+1+\alpha \right) \lambda}}=0.
	$$

\end{theorem}

\begin{proof}
	
	$\left( 1 \right) \Rightarrow \left( 2 \right) 	$	
	If $\mu$ is a vanishing $(\lambda, \alpha)$-Carleson measure. Let $t>0$ and $$  f_a\left( z \right) =\frac{\rho \left( a \right) ^{t/q}}{\rho \left( z,a \right) ^{\left( n+1+\alpha \right) /p+t/q}},   $$  We calculate to obtain that $$\mathop{\text{sup}}_{a\in T_B}\lVert f_a\left( z \right) \rVert _{p,\alpha}<\infty. $$
	
	So, we have $$
	\int_{T_B}{\frac{\rho \left( a \right) ^t}{\left| \rho \left( z,a \right) \right|^{\left( n+1+\alpha  \right) \lambda +t}}}d\mu \left( z \right) \rightarrow 0
	$$
	as $a \rightarrow \partial \widehat{T_B}.$
	
	We denote $Q_k=\overline{D\left( \mathbf{i},r \right) }$, if we can prove that $f_a$ converges uniformly to 0 on every $Q_k$ as $a\rightarrow \partial \widehat{T_B}$, then the conclusion holds.
	For a given $Q_k$, we have $$
	\mathop{\text{sup}}_{z\in Q_k}\left| f_a\left( z \right) \right|\lesssim \frac{\rho \left( a \right) ^{t/q}}{\left| \rho \left( \mathbf{i},a \right) \right|^{\left( n+1+\alpha  \right) /p+t/q}}.
	$$
	
	Since $\left| \rho \left( \mathbf{i},a \right) \right|\ge 1/2$ for all $a \in T_B$, it follows that $\mathop{\text{sup}}_{z\in Q_k}\left| f_a\left( z \right) \right|\lesssim  \rho(z) ^{t/q}$. Hence $f_a \rightarrow 0$ uniformly on $Q_k$ as $a \rightarrow bT_B.$
	Due to Lemma \ref{lem zitui}, therefore $f_a \rightarrow 0$ uniformly on $Q_k$ as $\left| a \right|\rightarrow \infty .$
	Thus, the condition $(2)$ holds.
	\\
	
	$\left( 2 \right) \Rightarrow \left( 3 \right) 	$	
	The inequality
	\[
	\int_{D(a,r)}{\frac{\rho(a)^t}{\rho(z,a)^{(n+1+\alpha)\lambda+t}}}d\mu(z) \le \int_{T_B}{\frac{\rho(a)^t}{\rho(z,a)^{(n+1+\alpha)\lambda+t}}}d\mu(z) \rightarrow 0
	\]
	combined with Lemma \ref{lem dengjia} yields the desired result.
	\\
	
	$\left( 3 \right) \Rightarrow \left( 1 \right) 	$	
	Assuming $\{f_j\}$ is a bounded sequence in $A_{\alpha}^{p}\left( T_B \right)$ and uniformly converges to 0 on every compact subset of $T_B$, according to the definition of vanishing Carleson measure, it suffices to prove the following statement: $$
	\lim_{j\rightarrow \infty} \int_{T_B}{\left| f_j\left( z \right) \right|^qd\mu \left( z \right)}=0.
	$$
	
	For any $r>0$ we fix an r-lattice $\{a_j\}$ in the Bergman metric. Since $a_j \rightarrow \partial \widehat{T_B}$
	as $j \rightarrow \infty$, 	$$
	\lim_{j\rightarrow \infty} \frac{\mu \left( D\left( a_j,r \right) \right)}{\rho \left( a_j \right) ^{\left( n+1+\alpha \right) \lambda}}=0.
	$$
	
	Given $\varepsilon >0$, there is a positive integer $N_0$ such that $$\frac{\mu \left( D\left( a_j,r \right) \right)}{\rho \left( a_j \right) ^{\left( n+1+\alpha \right) \lambda}} < \varepsilon,\quad j\ge N_0.$$
	
	Therefore, the inequality to be proved can be divided into the following two parts:
	$$
	\begin{aligned}
	\int_{T_B}{\left| f_j\left( z \right) \right|^qd\mu \left( z \right)}&\lesssim \sum_{k=1}^{N_0-1}{\frac{\mu \left( D\left( a_k,r \right) \right)}{\rho \left( a_k \right) ^{\left( n+1+\alpha  \right) \lambda}}\left( \int_{D\left( a_k,2r \right)}{\left| f_j\left( z \right) \right|^pdV_{\alpha}\left( z \right)} \right) ^{\lambda}}\\
	&+\sum_{k=N_0}^{\infty}{\frac{\mu \left( D\left( a_k,r \right) \right)}{\rho \left( a_k \right) ^{\left( n+1+\alpha  \right) \lambda}}\left( \int_{D\left( a_k,2r \right)}{\left| f_j\left( z \right) \right|^pdV_{\alpha}\left( z \right)} \right) ^{\lambda}}.
	\end{aligned}
	$$

	The first inequality must converge to 0 as $j \rightarrow \infty$, since $\{f_j\}$ converges to 0 on compact subsets. The second inequality, due to $\lambda >1$, implies that
	\[
	\sum_{k=N_0}^{\infty}{\frac{\mu \left( D\left( a_k,r \right) \right)}{\rho \left( a_k \right) ^{\left( n+1+\alpha  \right) \lambda}}\left( \int_{D\left( a_k,2r \right)}{\left| f_j\left( z \right) \right|^pdV_{\alpha}\left( z \right)} \right) ^{\lambda}}\le \varepsilon N^{\lambda}\lVert f_j \rVert _{p}^{q}
	\]
	holds. 
	
	Combining the arbitrariness of $\varepsilon$, we conclude that $\mu$ is a vanishing $(\lambda,\alpha)$-Carleson measure.
\end{proof}

\begin{remark}
	The theorem in this section is a generalization of theorem $4.1$ and theorem $4.2$ in \cite{Li}. Theorem $\ref{xt1}$ and Theorem $\ref{xt2}$ in this section characterize the Carleson measure and the vanishing Carleson measure, respectively. They will characterize the boundedness and compactness of the operator in the subsequent theorems $\ref{dth 1}$ and $\ref{dth 2}$, respectively.
\end{remark}

\section{Boundedness and compactness of operators}
\ \ \ \
Now, we present the characterizations of the boundedness and compactness of Toeplitz operators and Berezin-type operators on Bergman spaces over tubular domains.
\\

First, we present two basic lemma \ref{lem xin 1} and \ref{lem xin 2}.
\\

Let's recall the definition of a separated sequence. A sequence of points $\{z_k\}$ in $T_B$ is called a separated sequence if there exists $\delta>0$, such that $\beta{(z, w)}>\delta$ for any $i\ne j$.

\begin{lemma}\label{lem xin 1}
	Given a separated sequence $\{z_k\}$ in $T_B$, with $n<t<s$, for any $z\in T_B$, the following holds.
	
	\[
	\sum_{k=1}^{\infty}{\frac{\rho \left( z_k \right) ^t}{|\rho \left( z,z_k \right) |^s}}\lesssim \rho \left( z \right) ^{t-s}, \ z\in T_B.
	\]
\end{lemma}

Based on Bergman's decomposition (lemma \ref{lem duliang}) and combined with Lemma \ref{lem jifendengshi}, we can derive the result. We omit it's proof here.\\

The following Lemma \ref{lem xin 2} provides a good estimate for the norm of functions in the spaces $L_{\alpha}^{p}\left( T_B \right)$.

\begin{lemma}\label{lem xin 2}
	Consider a positive sequence $\{c_j\}$ and a separated sequence $\{a_j\}$ in $T_B$. Let $b \in T_B$ satisfy
	\[ b > n\max\left(1,\frac{1}{p}\right)+\frac{1+\alpha}{p}. \]
	Suppose $f$ is a measurable function on $T_B$ such that
	\[ |f(z)| \le \sum_{j=1}^{\infty}{\frac{c_j}{|\rho(z,a_j)|^b}}. \]
	Then, $f \in L_\alpha^p(T_B)$, and we have the estimate
	\[ \lVert f \rVert_{p,\alpha}^{p} \lesssim \sum_{j=1}^{\infty}{\frac{c_{j}^{p}}{\rho(a_j)^{bp-(n+1+\alpha)}}}. \]
\end{lemma}

\begin{proof}
	
	When $0 < p \leq 1$, we have the inequality \[   |f\left( z \right) |^p\le \sum_{j=1}^{\infty}{\frac{c_j}{|\rho \left( z,a_j \right) |^{bp}}}.   \]  By Lemma \ref{lem jifendengshi}, it follows that \[   \begin{aligned}
	\int_{T_B}{\left| f\left( z \right) \right|^pdV_{\alpha}\left( z \right)}&\le \sum_{j=1}^{\infty}{c_{j}^{p}}\int_{T_B}{\frac{1}{|\rho \left( z,a_j \right) |^{bp}}}dV_{\alpha}\left( z \right) 
	\\
	&\lesssim \sum_{j=1}^{\infty}{\frac{c_{j}^{p}}{\rho \left( a_j \right) ^{bp-\left( n+1+\alpha \right)}}}.
	\end{aligned}  \] 
	
	For $p>1$, we choose $b > n + \frac{1+\alpha}{p}$. Let $p'$ be the conjugate exponent of $p$, such that $\frac{1}{p}+\frac{1}{p'}=1$. By Hölder's inequality and Lemma \ref{lem xin 1}, we obtain
	$$
	\begin{aligned}
	|f\left( z \right) |^p&=\left( \sum_{j=1}^{\infty}{\frac{c_j}{\left| \rho \left( z,a_j \right) \right|^b}} \right) ^p\\
	&\le \left( \sum_{j=1}^{\infty}{\frac{\rho \left( a_j \right) ^{b-\left( 1+\alpha \right) /p}}{\left| \rho \left( z,a_j \right) \right|^b}} \right) ^{p-1}\left( \sum_{j=1}^{\infty}{\frac{c_{j}^{p}\rho \left( a_j \right) ^{b\left( 1-p \right) +\left( 1+\alpha \right) /p'}}{\left| \rho \left( z,a_j \right) \right|^b}} \right)\\
	&\lesssim {\rho (z)} ^{-\left( 1+\alpha \right) /p'}\left( \sum_{j=1}^{\infty}{\frac{c_{j}^{p}\rho \left( a_j \right) ^{b\left( 1-p \right) +\left( 1+\alpha \right) /p'}}{\left| \rho \left( z,a_j \right) \right|^b}} \right) .
	\end{aligned}
	$$
	
	This implies
	$$
	\lVert f \rVert _{p,\alpha}^{p}\lesssim \sum_{j=1}^{\infty}{c_{j}^{p}}\rho \left( a_j \right) ^{b\left( 1-p \right) +\left( 1+\alpha \right) /p'}\int_{T_B}{\frac{\rho \left( z \right) ^{-\left( 1+\alpha \right) /p'}}{\left| \rho \left( z,a_j \right) \right|^b}}dV_{\alpha}\left( z \right).
	$$

	Since $\alpha-(1+\alpha)/p^{\prime}=(1+\alpha)/p-1>-1$, and
	$$
	b-(n+1+\alpha)+(1+\alpha) / p^{\prime}=s-n-(1+\alpha) / p>0,
	$$
	we have completed the proof of the lemma according to Lemma \ref{lem jifendengshi}.
\end{proof}\\

Due to the above lemmas \ref{lem xin 1} and \ref{lem xin 2}, combine with the characterization of Carleson measure. Now, we can prove the first main theorem \ref{dth 1} of this paper. The theorem mainly states that under the condition $p_1 \le p_2 $, the boundedness of Toeplitz operators $T_\mu^\xi$, Berezin-type operators $B_{\mu}^{\xi}$, and $\mu$ being a $(\lambda, \gamma)$-Carleson measure are equivalent, and it further states that their norms are mutually controlled.

\begin{theorem}\label{dth 1}
	Let $0<p_1 \le p_2 <\infty$, $-1<\alpha_1,\alpha_2,\xi<\infty$ and $\mu\in \mathcal{M}_+$. Suppose that $$
	n+1+\xi>n\max \left( 1,\frac{1}{p_i} \right) +\frac{1+\alpha _i}{p_i}, \quad  i=1,2.
	$$
	Let
	$$
	\lambda=1+\frac{1}{p_1}-\frac{1}{p_2}, \ \gamma=\frac{1}{\lambda}\left(\xi+\frac{\alpha_1}{p_1}-\frac{\alpha_2}{p_2}\right) .
	$$
	Then the following statements are equivalent:
\\
		
		$(1)$ $B_{\mu}^{\xi}$ is bounded from $A_{\alpha_{1}}^{p_{1}}(T_B)$ to $L_{\alpha_{2}}^{p_{2}}(T_B).$ 
		
		$(2)$ $T_\mu^\xi$ is bounded from $A_{\alpha_1}^{p_1}(T_B)$ to $A_{\alpha_2}^{p_2}(T_B)$.
		
		$(3)$ The measure $\mu$ is a $(\lambda, \gamma)$-Carleson measure.

	Moreover, we have	
	$$
	\|B_{\mu}^{\xi}\|_{A_{\alpha_{1}}^{p_{1}}\to L_{\alpha_{2}}^{p_{2}}}\simeq\|T_{\mu}^{\xi}\|_{A_{\alpha_{1}}^{p_{1}}\to A_{\alpha_{2}}^{p_{2}}}\simeq\|\mu\|_{\lambda,\gamma}.
	$$
\end{theorem}

\begin{proof}
	
	$\left( 1 \right) \Rightarrow \left( 2 \right) 	$
	The inequality $\left| T_{\mu}^{\xi}f \right|\le B_{\mu}^{\xi}f$ implies that $(2)$ holds trivially.
	\\
	
	$\left( 2 \right) \Rightarrow \left( 3 \right) 	$	
	Fix $a \in T_B$ and define $f_a\left( z \right) =\rho \left( z,a \right) ^{-\left( n+1+\xi \right)}$. Given the condition $(n+1+\xi)p_1>n+1+\alpha_1$, according to Lemma \ref{lem jifendengshi} we know that $f_a \in A_{\alpha_{1}}^{p_{1}}(T_B)$ with $$\lVert f_a \rVert _{p_1,\alpha _1}^{p_1}\lesssim \rho \left( a \right) ^{\left( n+1+\alpha _1 \right) -\left( n+1+\xi \right) p_1}
	.$$
	
	Since $$
	\begin{aligned}
	T_{\mu}^{\xi}f_a\left( z \right) &=\int_{T_B}{\frac{f_a\left( w \right)}{\rho \left( z,w \right) ^{n+1+\xi}}d\mu \left( w \right)}
	\\
	&=\int_{T_B}{\frac{d\mu \left( w \right)}{\rho \left( z,w \right) ^{n+1+\xi}\rho \left( w,a \right) ^{n+1+\xi}}},
	\end{aligned}
	$$
	we have $$
	T_{\mu}^{\xi}f_z\left( z \right) =\int_{T_B}{\frac{d\mu \left( w \right)}{\left| \rho \left( z,w \right) \right|^{2\left( n+1+\xi \right)}}}\gtrsim \frac{\mu \left( D\left( z,r \right) \right)}{\rho \left( z \right) ^{2\left( n+1+\xi  \right)}}
	\\
	.
	$$

	Moreover, using the pointwise estimate for functions in Bergman spaces and the boundedness of the Toeplitz operator $T_\mu^\xi$, 
	we obtain
	$$
	\begin{aligned}
	T_{\mu}^{\xi}f_z\left( z \right) =\left| T_{\mu}^{\xi}f_z\left( z \right) \right|&\le \frac{\lVert T_{\mu}^{\xi}f_z \rVert _{p_2,\alpha _2}}{\rho \left( z \right) ^{\left( n+1+\alpha _2 \right) /p_2}}
	\\
	&\le \lVert T_{\mu}^{\xi} \rVert \lVert f_z \rVert _{p_1,a_1}\rho \left( z \right) ^{-\left( n+1+\alpha _2 \right) /p_2}
	\\
	&\lesssim \lVert T_{\mu}^{\xi} \rVert \rho \left( z \right) ^{\left( n+1+\alpha _1 \right) /p_1-\left( n+1+\alpha _2 \right) /p_2-\left( n+1+\xi \right)}.
	\\
	\end{aligned}
	$$
	
	Hence, we have
	$$
	\begin{aligned}
	\mu \left( D\left( z,r \right) \right) &\lesssim \lVert T_{\mu}^{\xi} \rVert \rho \left( z \right) ^{\left( n+1+\xi \right) +\left( n+1+\alpha _1 \right) /p_1-\left( n+1+\alpha _2 \right) /p_2}
	\\
	&=\rho \left( z \right) ^{\left( n+1+\gamma \right) \lambda}.
	\end{aligned}
	$$
	
	This implies, by Theorem \ref{xt1}, that $\mu$ is a $(\lambda,\gamma)$-Carleson measure with $$
	\lVert \mu \rVert _{\lambda ,\gamma}\lesssim \lVert T_{\mu}^{\xi} \rVert .
	$$
	\\
	$\left( 3 \right) \Rightarrow \left( 1 \right) 	$		
	Suppose that $\mu$ is a $(\lambda,\gamma)$-Carleson measure. Since the condition $0<p_1\le p_2<\infty$ implies that $\lambda >1$ and $p_2/p_1$$\ge1$.  According to Lemma \ref{lem dengjia}, we know that when $w \in D(a_j,r)$, we have $\left| \rho \left( z,w \right) \right|\simeq \left| \rho \left( z,a_j \right) \right|$. 
	
	Therefore, we obtain $$
	\begin{aligned}
	\left( B_{\mu}^{\xi}f \right) \left( z \right) &=\int_{T_B}{\frac{|f\left( w \right) |}{\left| \rho \left( z,w \right) \right|^{n+1+\xi}}}\text{d}\mu \left( w \right) 
	\\
	&\lesssim \sum_{j=1}^{\infty}{\int_{D\left( a_j,r \right)}{\frac{|f\left( w \right) |}{\left| \rho \left( z,w \right) \right|^{n+1+\xi}}}}d\mu \left( w \right) 
	\\
	&\lesssim \sum_{j=1}^{\infty}{\left( \mathop{\text{sup}}_{w\in D\left( a_j,r \right)}\left| f\left( w \right) \right| \right) \int_{D\left( a_j,r \right)}{\frac{1}{\left| \rho \left( z,w \right) \right|^{n+1+\xi}}}}d\mu \left( w \right) 
	\\
	&\lesssim \sum_{j=1}^{\infty}{\left( \mathop{\text{sup}}_{w\in D\left( a_j,r \right)}\left| f\left( w \right) \right| \right) \frac{\mu \left( D_j \right)}{\left| \rho \left( z,a_j \right) \right|^{n+1+\xi}}}.
	\end{aligned}
	$$
	
	According to Lemma \ref{lem ci tiao he fang da}, we have $$
	\left| f\left( w \right) \right|^{p_1}\le \frac{C}{\rho \left( a_j \right) ^{n+1+\alpha _1}}\int_{D\left( a_j,2r \right)}{\left| f\left( w \right) \right|^{p_1}dV_{\alpha _1}\left( w \right)}
	$$ holds for any $w \in D(a_j,r)$.\\	
	
	We get that 
	$$
	\left| \left( B_{\mu}^{\xi}f \right) \left( z \right) \right|\lesssim \sum_{j=1}^{\infty}{\left( \frac{1}{\rho \left( a_j \right) ^{n+1+\alpha _1}}\int_{D\left( a_j,2r \right)}{\left| f\left( w \right) \right|^{p_1}dV_{\alpha _1}\left( w \right)} \right) ^{1/p_1}}\frac{\mu \left( D\left( a_j,r \right) \right)}{\left| \rho \left( z,a_j \right) \right|^{n+1+\xi }}.
	$$
	
	According to Lemma \ref{lem xin 2}, we can obtain $$
	\lVert B_{\mu}^{\xi}f \rVert _{p_2,\alpha _2}^{p_2}\lesssim \sum_{j=1}^{\infty}{\left( \int_{D\left( a_j,2r \right)}{\left| f\left( w \right) \right|^{p_1}dV_{\alpha _1}\left( w \right)} \right) ^{p_2/p_1}}\left( \frac{\mu \left( D\left( a_j,r \right) \right)}{\left| \rho \left( a_j \right) \right|^{\left( n+1+\gamma  \right) \lambda}} \right) ^{p_2}.
	$$
	
	it follows Theorem \ref{xt1}, we have 
	$$
	\begin{aligned}
	\lVert B_{\mu}^{\xi}f \rVert _{p_2,\alpha _2}^{p_2}&\lesssim \lVert \mu \rVert _{\lambda ,\gamma}^{p_2}\sum_{j=1}^{\infty}{\left( \int_{D\left( a_j,2r \right)}{\left| f\left( z \right) \right|^{p_1}dV_{\alpha _1}\left( z \right)} \right) ^{p_2/p_1}}
	\\
	&\lesssim \lVert \mu \rVert _{\lambda ,\gamma}^{p_2}\left( \sum_{j=1}^{\infty}{\int_{D\left( a_j,2r \right)}{\left| f\left( z \right) \right|^{p_1}dV_{\alpha _1}\left( z \right)}} \right) ^{p_2/p_1}
	\\
	&\lesssim \lVert \mu \rVert _{\lambda ,\gamma}^{p_2}\lVert f \rVert _{p_1,\alpha _1}^{p_2}.
	\end{aligned}
	$$

	Hence, $B_\mu^\xi$ is bounded from $A_{\alpha_{1}}^{p_{1}}(T_B)$ to $L_{\alpha_{2}}^{p_{2}}(T_B).$ 
	
\end{proof}

Next, by combining the characterization of the vanishing Carleson measures in Theorem \ref{xt2}, we prove the second main theorem \ref{dth 2} of this paper. Before that, we present a lemma \ref{lem jin suan zi chong fen tiao jian}, which provides a sufficient condition for an operator to be compact.

\begin{lemma}\label{lem jin suan zi chong fen tiao jian}
	Let $0 < p_1, p_2 < \infty$,\ $-1 < \alpha_1,  \alpha_2, \xi < \infty$ and $\mu \in \mathcal{M}_+$. Assume that $T_\mu^\xi: A_{\alpha_1}^{p_1}(T_B) \rightarrow A_{\alpha_2}^{p_2}(T_B)$ is a bounded operator. Suppose that for every bounded sequence $\{f_k\}$ in $A_{\alpha_1}^{p_1}(T_B)$ such that $f_k \rightarrow 0$ uniformly on every compact subset of $T_B$ as $k \rightarrow \infty$, we have
	$$
	\lim _{k \rightarrow \infty}\left\|T_\mu^\xi f_k\right\|_{p_2, \alpha_2}=0 .
	$$
	Then, $T_\mu^\xi$ is compact from $A_{\alpha_1}^{p_1}(T_B)$ to $L_{\alpha_2}^{p_2}(T_B)$.
\end{lemma}	

\begin{proof}
	Let $\left\{f_k\right\}$ be a bounded sequence in $A_{\alpha_1}^{p_1}(T_B)$. Thus, there exists a constant $M>0$ such that $$\left\|f_k\right\|_{p_1, \alpha_1} \leq M$$ for all $k \geq 1$. 
	
	By Lemma \ref{lem ci tiao he fang da}, $\left\{f_k\right\}$ is uniformly bounded on every compact subset of $T_B$. 
	
	By Montel's Theorem, there exists a subsequence of $\left\{f_{k_j}\right\}, j=1,2,3 \ldots$, such that $f_{k_j} \rightarrow f$ uniformly on every compact subset of $T_B$ for some holomorphic function $f$ on $T_B$, as $j \rightarrow \infty$. 
	
	Applying Fatou's Lemma we have
	$$
	\begin{aligned}
	\int_{T_B}|f(z)|^{p_1} \mathrm{~d} V_{\alpha_1}(z) & =\int_{T_B} \lim _{j \rightarrow \infty}\left|f_{k_j}(z)\right|^{p_1} \mathrm{~d} V_{\alpha_1}(z) \\
	& \leq \lim _{j \rightarrow \infty} \int_{T_B}\left|f_{k_j}(z)\right|^{p_1} \mathrm{~d} V_{\alpha_1}(z) \\
	& \leq \lim _{j \rightarrow \infty}\left\|f_{k_j}\right\|_{p_1, \alpha_1}^{p_1} \leq M .
	\end{aligned}
	$$
	
	Thus, we have $$f \in A_{\alpha_1}^{p_1}(T_B).$$ 
	
	Therefore, $f_{k_j}-f \rightarrow 0$ uniformly on every compact subset of $T_B$ as $j \rightarrow \infty$. 
	
	Given our assumption, we have
	$$
	\lim _{j \rightarrow \infty}\left\|T_\mu^\xi\left(f_{k_j}-f\right)\right\|_{p_2, \alpha_2}=0 .
	$$
	
	It is straightforward to verify that
	$$
	\left\|T_\mu^\xi f_{k_j}-T_\mu^\xi f\right\|_{p_2, \alpha_2} \leq\left\|T_\mu^\xi\left(f_{k_j}-f\right)\right\|_{p_2, \alpha_2} .
	$$

	From this inequality, we obtain that
	$$
	\lim _{j \rightarrow \infty}\left\|T_\mu^\xi f_{k_j}-T_\mu^\xi f\right\|_{p_2, \alpha_2}=0,
	$$
	implying that $T_\mu^\xi f \in L_{\alpha_2}^{p_2}(T_B)$. 

	This implies that a subsequence of $\{T_{\mu}^{\xi}f_{k_j}\}$ converges in $L_{\alpha_2}^{p_2}(T_B)$, stating the compactness of $T_\mu^\xi$ from $A_{\alpha_1}^{p_1}(T_B)$ to $L_{\alpha_2}^{p_2}(T_B)$.
\end{proof}\\

The following Lemma \ref{lem ruo shou lian de deng jia tiao jian} is a well-known result, and its proof can also be found in \cite{Li}.

\begin{lemma}\label{lem ruo shou lian de deng jia tiao jian}
	Suppose $\left\{ f_j \right\}$ is a sequence in $A_{\alpha}^{p}\left( T_B \right)$ with $1<p<\infty.$ Then $f_j\rightarrow 0$ weakly in $A_{\alpha}^{p}\left( T_B \right)$ as $j\rightarrow \infty$ if and only if $\left\{ f_j \right\}$ is bounded in $A_{\alpha}^{p}\left( T_B \right)$ and converges to $0$ uniformly on each compact subset of $T_B$.
\end{lemma}

With the above lemmas \ref{lem jin suan zi chong fen tiao jian} and \ref{lem ruo shou lian de deng jia tiao jian}, we can now prove the main theorem \ref{dth 2}.

\begin{theorem}\label{dth 2}
	Let $1<p_1 \le p_2 <\infty$, $-1<\alpha_1,\alpha_2,\xi<\infty$ and let $\mu\in \mathcal{M}_+$. Suppose that $$
	n+1+\xi >n\max \left( 1,\frac{1}{p_i} \right) +\frac{1+\alpha _i}{p_i}, \quad  i=1,2.
	$$
	Let
	$$
	\lambda=1+\frac{1}{p_1}-\frac{1}{p_2}, \ \gamma=\frac{1}{\lambda}\left(\xi+\frac{\alpha_1}{p_1}-\frac{\alpha_2}{p_2}\right) .
	$$
	Then the following statements are equivalent:
\\
		
		$(1)$ $B_{\mu}^{\xi}$ is compact from $A_{\alpha_{1}}^{p_{1}}(T_B)$ to $L_{\alpha_{2}}^{p_{2}}(T_B).$ 
		
		$(2)$ $T_\mu^\xi$ is compact from $A_{\alpha_1}^{p_1}(T_B)$ to $A_{\alpha_2}^{p_2}(T_B)$.
		
		$(3)$ The measure $\mu$ is a vanishing $(\lambda, \gamma)$-Carleson measure.

\end{theorem}

\begin{proof}
	
	$\left( 1 \right) \Rightarrow \left( 2 \right) 	$		
	The inequality $\left| T_{\mu}^{\gamma}f \right|\le B_{\mu}^{\gamma}f$ implies that $(2)$ holds trivially.	
	\\

	$\left( 2 \right) \Rightarrow \left( 3 \right) 	$		
	Let $T_\mu^\xi$ be a compact linear map, and consider $\left\{ a_k \right\} \in T_B$ such that $\left| a_k \right|\rightarrow \partial \widehat{T_B}$ as $k \rightarrow \infty$. 
	
	We consider the function
	$$
	f_k\left( z \right) =\frac{\rho \left( z \right) ^{\left( n+1+\xi \right) -\left( n+1+\alpha _1 \right) /p_1}}{\rho \left( z,a_k \right) ^{n+1+\xi}}.
	$$
	By Lemma \ref{lem jifendengshi}, we have $\mathop{\text{sup}}_k\lVert f_k \rVert _{p_1,\alpha _1}<\infty $. 
	
	Let $Q_k=\overline{D\left( \mathbf{i},r \right) }$ and by Lemma \ref{lem zitui}, we know that $f_k$ uniformly converges to 0 on the compact subset $ Q_k $ of $T_B$. 
	\\
	
	Since
	$$
	\begin{aligned}
	\frac{\mu \left( D\left( a_k,r \right) \right)}{\rho \left( a_k \right) ^{\left( n+1+\gamma  \right) \lambda}}
	&\lesssim \rho \left( a_k \right) ^{\left( n+1+\alpha _2 \right) /p_2}\int_{T_B}{\frac{\rho \left( a_k \right) ^{\left( n+1+\xi  \right) -\left( n+1+\alpha _1 \right) /p_1}}{\left| \rho \left( w,a_k \right) \right|^{2\left( n+1+\xi  \right)}}d\mu \left( w \right)}
	\\
	&\simeq \rho \left( a_k \right) ^{\left( n+1+\alpha _2 \right) /p_2}T_{\mu}^{\xi}f_k\left( a_k \right) ,
	\end{aligned}
	$$
	by lemma \ref{lem ci tiao he fang da} we obtain
	$$
	T_{\mu}^{\xi}f_k\left( a_k \right) \lesssim \frac{1}{\rho \left( a_k \right) ^{\left( n+1+\alpha _2 \right) /p_2}}\left( \int_{D\left( a_k,r \right)}{\left| T_{\mu}^{\xi}f_k\left( z \right) \right|^{p_2}dV_{\alpha _2}\left( z \right)} \right) ^{1/p_2}.
	$$
	\\
	
	Since $T_\mu^\xi$ is compact, Lemma \ref{lem ruo shou lian de deng jia tiao jian} implies that $f_k(z)$ weakly converges to 0. 
	
	Therefore, we obtain $ \lVert T_{\mu}^{\xi}f_k \rVert _{p_2,\alpha _2}\rightarrow 0. $	
	
	 Hence, we have $${\mu \left( D\left( a_k,r \right) \right)}/{\rho \left( a_k \right) ^{\left( n+1+\gamma  \right) \lambda}} \rightarrow 0.$$ By Theorem \ref{xt2}, $\mu$ is a vanishing $(\lambda, \gamma)$-Carleson measure.	
	\\
	
	$\left( 3 \right) \Rightarrow \left( 1 \right) 	$		
	Given the conditions, $\mu$ is a vanishing $(\lambda,\gamma)$-Carleson measure. According to Lemma \ref{lem jin suan zi chong fen tiao jian}, we only need to prove that for any bounded sequence $\left\{f_k\right\}$ in $A_{\alpha_1}^{p_1}(T_B)$ that uniformly converges to 0 on compact subsets of $T_B$, the fact $ \lVert B_{\mu}^{\xi}f_k \rVert _{p_2,\alpha _2}\rightarrow 0$ holds.	
	
	According to the proof of $(3)\Rightarrow(1)$ in Theorem \ref{dth 1}, we establish the following equation: 	
	$$
	\lVert B_{\mu}^{\xi}f_k \rVert _{p_2,\alpha _2}^{p_2}\lesssim \sum_{j=1}^{\infty}{\left( \int_{D\left( a_j,2r \right)}{\left| f_k\left( w \right) \right|^{p_1}dV_{\alpha _1}\left( w \right)} \right) ^{p_2/p_1}}\left( \frac{\mu \left( D\left( a_j,r \right) \right)}{\left| \rho \left( a_j \right) \right|^{\left( n+1+\gamma  \right) \lambda}} \right) ^{p_2}.
	$$
	
	Since $p_2/p_1\ge1$ and $\left\{f_k\right\}$ is a bounded sequence in $A_{\alpha_{1}}^{p_{1}}(T_B)$, it follows that:
	$$
	\mathop{\text{sup}}_k\sum_{j=1}^{\infty}{\left( \int_{D\left( a_j,2r \right)}{\left| f_k\left( w \right) \right|^{p_1}dV_{\alpha _1}\left( w \right)} \right) ^{p_2/p_1}}\lesssim \mathop{\text{sup}}_k\lVert f_k \rVert _{p_1,\alpha _1}^{p_2}<\infty .
	$$
	
	Let $E_n=T_B-Q_n$ and $Q_n=\overline{D\left( \mathbf{i},n \right) }$ be a compact subset in $T_B$. 
	
	Since $\left\{f_k\right\}$ converges to 0 uniformly on compact subsets of $T_B$ as $k \rightarrow \infty$, it follows that:
	$$
	\lim_{k\rightarrow \infty} \sum_{j:a_j\in Q_n}{\left( \int_{D\left( a_j,2r \right)}{\left| f_k\left( w \right) \right|^{p_1}dV_{\alpha _1}\left( w \right)} \right) ^{p_2/p_1}}=0.
	$$
	
	According to Theorem \ref{xt2}, we have 
	$$
	\lim_{|a_j|\rightarrow \partial \widehat{T_B}} \frac{\mu \left( D\left( a_j,r \right) \right)}{\rho \left( a_j \right) ^{\left( n+1+\alpha \right) \lambda}}=0.
	$$
	
	So, for any $\varepsilon>0$, there is an $n_0 \in \mathbb{N}$, such that 
	$$
	\left| \frac{\mu \left( D\left( a_j,r \right) \right)}{\rho \left( a_j \right) ^{\left( n+1+\alpha \right) \lambda}} \right|<\varepsilon
	$$ 
	for all $a_j\in E_{n_0}$. 
	\\
	
	Therefore 
	$$
	\begin{aligned}
	\lVert B_{\mu}^{\xi}f_k \rVert _{p_2,\alpha _2}^{p_2}&\le \sum_{j:j\in Q_{n_0}}{\left| \left( \int_{D\left( a_j,2r \right)}{\left| f_k\left( w \right) \right|^{p_1}dV_{\alpha _1}\left( w \right)} \right) ^{p_2/p_1}\left( \frac{\mu \left( D\left( a_j,r \right) \right)}{\left| \rho \left( a_j \right) \right|^{\left( n+1+\gamma  \right) \lambda}} \right) ^{p_2} \right|}
	\\
	&+\sum_{j:j\in E_{n_0}}{\left| \left( \int_{D\left( a_j,2r \right)}{\left| f_k\left( w \right) \right|^{p_1}dV_{\alpha _1}\left( w \right)} \right) ^{p_2/p_1}\left( \frac{\mu \left( D\left( a_j,r \right) \right)}{\left| \rho \left( a_j \right) \right|^{\left( n+1+\gamma  \right) \lambda}} \right) ^{p_2} \right|}
	\\
	&\le \mathop{\text{sup}}_{j:j\in Q_{n_0}}\left| \left( \frac{\mu \left( D\left( a_j,r \right) \right)}{\left| \rho \left( a_j \right) \right|^{\left( n+1+\gamma  \right) \lambda}} \right) ^{p_2} \right|\sum_{j:j\in Q_{n_0}}{\left| \left( \int_{D\left( a_j,2r \right)}{\left| f_k\left( w \right) \right|^{p_1}dV_{\alpha _1}\left( w \right)} \right) ^{p_2/p_1} \right|}
	\\
	&+\varepsilon \sum_{j:j\in E_{n_0}}{\left| \left( \int_{D\left( a_j,2r \right)}{\left| f_k\left( w \right) \right|^{p_1}dV_{\alpha _1}\left( w \right)} \right) ^{p_2/p_1} \right|}.
	\end{aligned}
	$$
	
	Letting $\varepsilon\rightarrow0$ and then letting $k\rightarrow\infty$, we have the above the result.
\end{proof}\\

Finally, we used a specific method to prove the last main theorem \ref{dth 3} in the present paper. The method used in the following theorem was pioneered by Luecking and utilizes the Khinchine's inequality. The detailed description is as follows.

Define the Rademacher functions $r_k$ by
$$
\begin{aligned}
& r_0(t)=\left\{\begin{array}{rl}
1 & 0 \leq t-[t]<\frac{1}{2}, \\
-1 & \frac{1}{2} \leq t-[t]<1 ;
\end{array}\right. \\
& r_n(t)=r_0\left(2^n t\right), \quad n>0 .
\end{aligned}
$$
Here $\left[ t \right] $ denotes the largest integer not greater than $t$.

Then Khinchine's inequality is the following.

For $0<p<\infty$ there exist constants $0<A_p \leq B_p<$ $\infty$ such that, for all natural numbers $m$ and all complex numbers $c_1, c_2, \ldots$, $c_m$, we have
$$
A_p\left(\sum_{j=1}^m\left|c_j\right|^2\right)^{p / 2} \leq \int_0^1\left|\sum_{j=1}^m c_j r_j(t)\right|^p d t \leq B_p\left(\sum_{j=1}^m\left|c_j\right|^2\right)^{p / 2} .
$$

Theorem \ref{dth 3} states that under the assumption of $p_2 < p_1$, the boundedness and compactness of Toeplitz operators are unified, and a connection with an integrable function is established.

\begin{theorem}\label{dth 3}
	Let $0<p_2 < p_1 <\infty$, $-1<\alpha_1,\alpha_2,\xi<\infty$ and let $\mu\in \mathcal{M}_+$. Suppose that $$
	n+1+\xi >n\max \left( 1,\frac{1}{p_i} \right) +\frac{1+\alpha _i}{p_i}, \quad  i=1,2.
	$$
	Let
	$$
	\lambda=1+\frac{1}{p_1}-\frac{1}{p_2}, \ \gamma=\frac{1}{\lambda}\left(\xi+\frac{\alpha_1}{p_1}-\frac{\alpha_2}{p_2}\right) .
	$$
	Then the following statements are equivalent:
\\
	
		$(1)$ $T_{\mu}^{\xi}$ is bounded from $A_{\alpha_{1}}^{p_{1}}(T_B)$ to $A_{\alpha_{2}}^{p_{2}}(T_B).$ 
		
		$(2)$
		$
		\left\{ \frac{\mu \left( D\left( a_k,r \right) \right)}{\rho \left( a_k \right) ^{\left( n+1+\gamma  \right) \lambda}} \right\} \in l^{1/\left( 1-\lambda \right)}$ for any $r>0$ and for any r-lattice $\left\{a_k\right\}$.
		
		$(3)$ $T_\mu^\xi$ is compact from $A_{\alpha_1}^{p_1}(T_B)$ to $A_{\alpha_2}^{p_2}(T_B)$.

\end{theorem}

\begin{proof}
	$\left( 1 \right) \Rightarrow \left( 2 \right) 	$	
	Let $r_k(t)$ be a sequence of Rademacher functions and $\left\{a_k\right\}$ be an $r$-lattice on $T_B$. 
	
	Since $$n+1+\xi >n\max \left( 1,\frac{1}{p_1} \right) +\frac{1+\alpha _1}{p_1},$$ 
	it follows from \cite[Theorem 2]{coifman1980representation} that for any sequence $\left\{\lambda_k\right\} \in l^{p_1}$, \\the function
	$$
	f_t(z)=\sum_{k=1}^{\infty} \lambda_k r_k(t) \frac{\rho\left(a_k\right)^{n+1+\xi-(n+1+\alpha_1) / p_1}}{\rho\left(z, a_k\right)^{n+1+\xi}}
	$$
	is in $A_{\alpha_1}^{p_1}(T_B)$ with$\left\|f_t\right\|_{p,\alpha_1} \lesssim\left\|\left\{\lambda_k\right\}\right\|_{l^p}$ for almost every $t \in (0,1)$. 
	
	Denote
	$$
	f_k(z)=\frac{\rho\left(a_k\right)^{n+1+\xi-(n+1+\alpha_1) / p_1}}{\rho\left(z, a_k\right)^{n+1+\xi}}.
	$$
	\\
	
	Since $T_\mu^\xi$ is bounded from $A_{\alpha_{1}}^{p_{1}}(T_B)$ to $A_{\alpha_{2}}^{p_{2}}(T_B)$. 
	\\

	it follows that
	$$
	\begin{aligned}
	\lVert T_{\mu}^{\xi}f_t \rVert _{p_2,\alpha _2}^{p_2}&=\int_{T_B}{\left| \sum_{k=1}^{\infty}{\lambda _k}r_k\left( t \right) T_{\mu}^{\xi}f_k\left( z \right) \right|^{p_2}dV_{\alpha _2}\left( z \right)}
	\\
	&\lesssim \lVert T_{\mu}^{\xi} \rVert ^{p_2}\lVert f_t \rVert _{p_2,\alpha _2}^{p_2}\lesssim \lVert T_{\mu}^{\xi} \rVert ^{p_2}\left( \sum_{k=1}^{\infty}{\left| \lambda _k \right|^{p_1}} \right) ^{p_2/p_1}.
	\end{aligned}
	$$
	\\
	
	Integrating both sides with respect to $t$ from 0 to 1, then using Fubini's theorem and Khinchine's inequality, we obtain
	$$
	\int_{T_B}{\left( \sum_{k=1}^{\infty}{\left| \lambda _k \right|^2}\left| T_{\mu}^{\xi}f_k\left( z \right) \right|^2 \right) ^{p_2/2}dV_{\alpha _2}\left( z \right)}\lesssim \lVert T_{\mu}^{\xi} \rVert ^{p_2}\left( \sum_{k=1}^{\infty}{\left| \lambda _k \right|^{p_1}} \right) ^{p_2/p_1}.
	$$
	
	Through the Hölder inequality and Lemma \ref{lem duliang}, we have
	$$
	\begin{aligned}
	\sum_{k=1}^{\infty}{\left| \lambda _k \right|^{p_2}}&\int_{D\left( a_k,2r \right)}{\left| T_{\mu}^{\xi}f_k\left( z \right) \right|^{p_2}dV_{\alpha _2}\left( z \right)}\\
	&\le \max \left\{ 1,N^{1-p_2/2} \right\} \int_{T_B}{\left( \sum_{k=1}^{\infty}{\left| \lambda _k \right|^2}\left| T_{\mu}^{\xi}f_k\left( z \right) \right|^2 \right) ^{p_2/2}dV_{\alpha _2}\left( z \right)}\\
	&\lesssim \lVert T_{\mu}^{\xi} \rVert ^{p_2}\left( \sum_{k=1}^{\infty}{\left| \lambda _k \right|^{p_1}} \right) ^{p_2/p_1}.
	\end{aligned}
	$$

	In the proof of that $(2)\Rightarrow(3)$ in Theorem \ref{dth 2}, we establish the following equation: $$
	\frac{\mu \left( D\left( a_k,r \right) \right)}{\rho \left( a_k \right) ^{\left( n+1+\gamma  \right) \lambda}}\simeq \rho \left( a_k \right) ^{\left( n+1+\alpha _2 \right) /p_2}T_{\mu}^{\xi}f_k\left( a_k \right) .
	$$
	
	Based on Lemma \ref{lem ci tiao he fang da}, we obtain
	$$
	\sum_{k=1}^{\infty}{\left| \lambda _k \right|^{p_2}}\frac{\mu \left( D\left( a_k,r \right) \right)}{\rho \left( a_k \right) ^{\left( n+1+\gamma  \right) \lambda}}\lesssim \lVert T_{\mu}^{\xi} \rVert ^{p_2}\left( \sum_{k=1}^{\infty}{\left| \lambda _k \right|^{p_1}} \right) ^{p_2/p_1}.
	$$
	
	Then the duality yields that 
	$$
	\left\{ \frac{\mu \left( D\left( a_k,r \right) \right)}{\rho \left( a_k \right) ^{\left( n+1+\gamma  \right) \lambda}} \right\} \in l^{1/\left( 1-\lambda \right)}.
	$$
	\\
	$\left( 2 \right) \Rightarrow \left( 3 \right) 	$
	In order to prove that $T_{\mu}^{\xi}$ is a compact mapping, according to Lemma \ref{lem jin suan zi chong fen tiao jian}, it suffices to show that for any bounded sequence $\left\{f_k\right\}$ in $A_{\alpha_{1}}^{p_{1}}(T_B)$ that uniformly converges to 0 on any compact subset of $T_B$, the property $\lVert T_{\mu}^{\xi}f_k \rVert _{p_2,\alpha _2}\rightarrow 0
	$ holds.
	
	Similar to the proof of $(3)\Rightarrow(1)$ in Theorem \ref{dth 1}, we establish the following equation
	$$
	\lVert T_{\mu}^{\xi}f_k \rVert _{p_2,\alpha _2}^{p_2}\lesssim \sum_{j=1}^{\infty}{\left( \int_{D\left( a_j,2r \right)}{\left| f_k\left( w \right) \right|^{p_1}dV_{\alpha _1}\left( w \right)} \right) ^{p_2/p_1}}\left( \frac{\mu \left( D\left( a_j,r \right) \right)}{\left| \rho \left( a_j \right) \right|^{\left( n+1+\gamma  \right) \lambda}} \right) ^{p_2}.
	$$

	Given the known conditions, we deduce that 
	$$
	\left\{ \frac{\mu \left( D\left( a_k,r \right) \right)}{\rho \left( a_k \right) ^{\left( n+1+\gamma  \right) \lambda}} \right\} \in l^{1/\left( 1-\lambda \right)}.
	$$
	\\
	
	Denote by 
	$$A=\frac{\mu \left( D\left( a_k,r \right) \right)}{\rho \left( a_k \right) ^{\left( n+1+\gamma  \right) \lambda}},\quad B_k=\int_{D\left( a_j,2r \right)}{\left| f_k\left( w \right) \right|^{p_1}dV_{\alpha _1}\left( w \right)},\quad E_k=T_B-Q_k,$$
	Where $Q_k=\overline{D\left( \mathbf{i},k \right) }.$
	\\
	
	Since $A\in l^{1/\left( 1-\lambda \right)}$, it follows that, for any $\varepsilon>0$, there is an $k_0 \in \mathbb{N}$, such that 
	$$
	\sum_{j:a_j\in E_{k_0}}{\left| A \right|^{1/\left( 1-\lambda \right)}}<\varepsilon .
	$$
	
	Since $Q_{k_0}$ is a compact subsets of $T_B$, so we have 
	$$
	\lim_{k\rightarrow \infty} \sum_{j:j\in Q_{k_0}}{\left| B_k\right|}=0.
	$$
	
	On the other hand, since $p_1>p_2$, by Hölder inequality, we have that
	
	$$
	\begin{aligned}
	\sum_{j=1}^{\infty}{A^{p_2}B_k ^{p_2/p_1}}&\le \sum_{j:j\in Q_{k_0}}{\left| A \right|^{p_2}\left| B_k \right|^{p_2/p_1}}+\sum_{j:a_j\in E_{k_0}}{\left| A \right|^{p_2}\left| B_k \right|^{p_2/p_1}}\\
	&\le \left( \sum_{j:j\in Q_{k_0}}{\left| A \right|^{1/\left( 1-\lambda \right)}} \right) ^{1-p_2/p_1}\left( \sum_{j:j\in Q_{k_0}}{\left| B_k \right|} \right) ^{p_2/p_1}\\
	&+\left( \sum_{j:j\in E_{k_0}}{\left| A \right|^{1/\left( 1-\lambda \right)}} \right) ^{1-p_2/p_1}\left( \sum_{j:j\in E_{k_0}}{\left| B_k \right|} \right) ^{p_2/p_1}\\
	&\le \left( \sum_{j:j\in Q_{k_0}}{\left| A \right|^{1/\left( 1-\lambda \right)}} \right) ^{1-p_2/p_1}\left( \sum_{j:j\in Q_{k_0}}{\left| B_k \right|} \right) ^{p_2/p_1}\\
	&+\varepsilon ^{1-p_2/p_1}\left( \sum_{j:j\in E_{k_0}}{\left|  B_k \right|} \right) ^{p_2/p_1}.
	\end{aligned}
	$$

	Letting $\varepsilon \rightarrow0$ and then letting $k\rightarrow\infty$, we obtain $\lVert T_{\mu}^{\xi}f_k \rVert _{p_2,\alpha _2}\rightarrow 0.$
	\\
	
	$\left( 3 \right) \Rightarrow \left( 1 \right) 	$
	Since compact operators are bounded operators, so the result holds trivially.
	
\end{proof}

\end{document}